\documentclass{BoTiWeZh19}

\title[On sequences associated to the invariant theory of Lie algebras]{On
sequences associated to the invariant theory \\ of rank two simple Lie
algebras}

\author[A. Bostan, J. Tirrell, B. W. Westbury, Y. Zhang]{Alin
Bostan\thanks{\href{mailto:alin.bostan@inria.fr}{alin.bostan@inria.fr}}\addressmark{1},
Jordan
Tirrell\thanks{\href{mailto:jtirrell@mtholyoke.edu}{jtirrell@mtholyoke.edu}}\addressmark{2}
Bruce W.
Westbury\thanks{\href{mailto:Bruce.Westbury@utdallas.edu}{Bruce.Westbury@utdallas.edu}.}\addressmark{3},
\and Yi Zhang\thanks{\href{mailto:zhangy@amss.ac.cn}{zhangy@amss.ac.cn}. Yi
Zhang was partially supported by the UTD Program P-1-03246, the Natural
Science Foundation of USA Grants CCF-1815108 and CCF-1708884.}\addressmark{3,
4}}

\address{\addressmark{1}Inria, Universit{\'e} Paris-Saclay, 1 rue Honor{\'e}
d'Estienne d'Orves, 91120 Palaiseau, France \\ \addressmark{2}Department of
Mathematics and Computer Science, Washington College, USA \\
\addressmark{3}Department of Mathematical Sciences, The Unversity of Texas at
Dallas, USA \\ \addressmark{4}Department of Mathematical Sciences, Xi'an
Jiaotong-Liverpool University, Suzhou, China}

\abstract{We study two families of sequences, listed in the On-Line
Encyclopedia of Integer Sequences (\href{https://oeis.org}{OEIS}), which are
associated to invariant theory of Lie algebras. For the first family, we prove
combinatorially that the sequences \oeis{A059710} and \oeis{A108307} are
related by a binomial transform. Based on this, we present two independent
proofs of a recurrence equation for \oeis{A059710}, which was conjectured by
Mihailovs. Besides, we also give a direct proof of Mihailovs' conjecture by
the method of algebraic residues. As a consequence, closed formulae for the
generating function of sequence~\oeis{A059710} are obtained in terms of
classical Gaussian hypergeometric functions. Moreover, we show that sequences
in the second family are also related by binomial transforms.}


\begin{document}

\maketitle

\section{Introduction} The representation theory of simple Lie algebras is a
keystone of algebraic and enumerative combinatorics, giving rise to
combinatorial objects such as tableaux, symmetric functions, quantum groups,
crystal graphs, and many others. We are interested in two families of
sequences in OEIS. Sequences in the first family of sequences are called
\emph{octant sequences}. For instance, sequence~\oeis{A059710}, which is the
first octant sequence, is defined to be a sequence associated to fundamental
representations of the exceptional simple Lie algebra~$G_2$, of rank two and
dimension fourteen~\cite{Kuperberg1997}. The second octant sequence
is~\oeis{A108307}~\cite{MR3779614}, and it is defined to be the cardinality of
the set of set partitions of $[n]$ with no enhanced 3-crossing.
Our first contribution is to prove that sequences~\oeis{A059710}
and~\oeis{A108307} are tightly related, by a binomial transform.

\begin{thm}\label{thm:bt} Let $T_3(n)$ and $E_3(n)$ be the $n$-th terms of~\oeis{A059710} and~\oeis{A108307}, respectively. 
Then $E_3$ is the binomial transform of $T_3$, \ie, for $n \geq 0$,
\[ E_3(n) = \sum_{k=0}^n \binom{n}{k}\,T_3(k). \]
\end{thm}
Theorem~\ref{thm:bt} provides an unexpected connection between invariant
theory of~$G_2$ and combinatorics of set partitions. In the same
spirit,~\cite[Theorem 1]{Jordan2019} proves a binomial relation between~$E_3$
and~\oeis{A108304}, which is the third octant sequence. Together, these two
results show that the octant sequences are associated to representations of
$G_2$.

Building on Theorem~\ref{thm:bt}, as well as on results by Bousquet-M{\'e}lou
and Xin~\cite{Bousquet-Melou2005}, our second contribution is to give two
independent proofs of a recurrence equation for $T_3$ conjectured by
Mihailovs~\cite[\S3]{Westbury2005}, which was the initial motivation for our
study:
\begin{thm}\label{thm:rc}
The sequence $T_3$ is determined by the initial conditions
$ T_3\left( 0 \right)=1$, $T_3\left( 1 \right)=0$, $T_3\left( 2 \right)=1$ and the recurrence relation that for
$n\ge 0$,
\begin{multline}
14\left( n+1 \right)  \left( n+2 \right) T_3 \left( n \right) +
 \left( n+2 \right) \left( 19n+75 \right)T_3 \left( n+1 \right) \\
+2 \left( n+2 \right) \left( 2n+11 \right) T_3\left( n+2 \right)
  - \left( n+8
 \right)  \left( n+9 \right) T_3\left( n+3 \right)=0.
\end{multline}
\end{thm}

Furthermore, we give an alternative proof of Theorem~\ref{thm:rc}, using the
interpretation of~$T_3$ in terms of $G_2$ walks and using algorithms for
computing Picard-Fuchs differential equations for algebraic residues. As a
consequence, closed formulae for the generating function of $T_3$ are obtained
in terms of the classical Gaussian hypergeometric function.

We consider a second family of sequences, called~\emph{quadrant sequences}.
These are defined to be sequences associated to representations of $G_2$
restricted to $SL(3)$. By invariant theory, these sequences are also are
related by binomial transforms. Based on this, we give a uniform recurrence
equation holding \emph{for each quadrant sequence}. Furthermore, we show that
sequences in the second family are identical to quadrant sequences because
they satisfy the same initial conditions and recurrence equations. They are
related to the octant sequences by the branching rules~\cite{MR638077} for the
maximal subgroup $SL(3)$ of $G_2$.

\section{Preliminaries}

In this section, we give a general construction of sequences associated to the
invariant theory of algebraic groups and the relations between them.

\begin{defn} \label{DEF:seq}
Let $G$ be a reductive complex algebraic group and let~$V$ be a (finite
dimensional) representation of~$G$. The {\emph sequence associated to $(G,
V)$}, denoted~$\ba_V$, is the sequence whose $n$-th term is the multiplicity
of the trivial representation in the tensor power $\otimes^nV$. 
\end{defn}

Given the sequence $\ba$ with $n$-th term $a(n)$, the \Dfn{binomial transform
of $\ba$} is the sequence, denoted $\bt \ba$, whose $n$-th term is
\[
\sum_{i=0}^n \binom{n}{i}a(i).
\]
Here, $\bt$ denotes the \emph{binomial transform operator}. Binomial
transforms arise naturally for sequences associated to the representations of
reductive complex algebraic groups.

\begin{lemma} \label{LEM:binomialrep} 
Assume that $\ba_V$ is the sequence associated to $(G, V)$ as specified in
Definition~\ref{DEF:seq}. Then $\ba_{V\oplus \bC} = \bt \ba_V$.
\end{lemma}

The binomial transform also arises naturally for lattice walks restricted to a
domain.

\begin{lemma} \label{LEM:binomialwalks}
Assume that a sequence $\ba$ enumerates walks in a lattice, confined to a
domain $D$, using a set of steps~$S$. Then $\bt \ba$ is given by adding a new
step corresponding to the zero element of the lattice; that is, $S$ is
replaced by the disjoint union $S\coprod \{0\}$ without changing the domain.
\end{lemma}

The binomial transform can also be regarded as an operator on the generating
function of a sequence. Let $G(t) = \sum_{n = 0}^\infty a(n) t^n$ be the
generating function of the sequence $\ba$. We denote the generating function
of $\bt^k \ba$ by $\bt^kG$. Then we have
\begin{lemma}\label{lem:btg}
For $k\in\bZ$, the $k$-th binomial transform of $G(t)$ is
\[ (\bt^kG)(t) = \frac 1{1-k\;t} G\left(\frac {t}{1-k\;t}\right). \]
\end{lemma}

\section{Octant sequences}

The first family of sequences we are interested in is illustrated in
Figure~\ref{fig:seq}. In Subsection~\ref{SUBSEC:latticewalks}, we show
that~\oeis{A059710} and~\oeis{A108307} are related by the binomial transform,
which is Theorem~\ref{thm:bt}. Combining with~\cite[Theorem 1]{Jordan2019} and
Lemma~\ref{LEM:binomialrep}, sequences in this first family are identical to
the sequences $\ba_{V\oplus k\,\bC}$ (Definition~\ref{DEF:seq}), where $V$ is
the seven dimensional fundamental representation of the exceptional simple Lie
algebra $G_2$, for $k = 0, 1, 2$. They can also be interpreted in terms of
lattice walks restricted to the octant. Thus, we call them \emph{octant
sequences}.

Let $T_3(n)$ be the $n$-th term of the first octant sequence~\oeis{A059710}.
In Subsection~\ref{SUBSEC:algebra}, based on Theorem~\ref{thm:bt}, we present
two independent proofs of~Theorem~\ref{thm:rc}, which gives a linear
recurrence equation for~$T_3(n)$. Besides, we also give a direct proof
of~Theorem~\ref{thm:rc} by the method of algebraic residues. As a consequence,
closed formulae for the generating function of~$T_3$ are presented in terms of
hypergeometric functions in Subsection~\ref{SUBSEC:closedformulae}.

\begin{figure}
\begin{center}
\begin{tabular}{c|rrrrrrrrrr}
$\ba$ (\href{https://oeis.org}{OEIS} tag) & 0 & 1 & 2 & 3 & 4 & 5 & 6 & 7 & 8 & 9 \\ \hline
\oeis{A059710} & 1 & 0 & 1 & 1 & 4 & 10 & 35 & 120 & 455 & 1792 \\
\oeis{A108307} & 1 & 1 & 2 & 5 & 15 & 51 & 191 & 772 & 3320 & 15032 \\
\oeis{A108304} & 1 & 2 & 5 & 15 & 52 & 202 & 859 & 3930 & 19095 & 97566
\end{tabular}
\caption{The first family of sequences $\ba$, with their \href{https://oeis.org}{OEIS} tags, and their first terms.}
 \label{fig:seq}
\end{center}
\end{figure}

\subsection{Lattice walks} \label{SUBSEC:latticewalks}
In this subsection we prove Theorem~\ref{thm:bt}, using lattice walks. An
\Dfn{excursion} is a walk which starts and ends at the origin. The main idea
of our proof is to exhibit excursions that are enumerated by the two sequences
and then to compare these.

Sequence $T_3$ has an interpretation via enumerating excursions on the weight
lattice of~$G_2$. This interpretation is given in
\cite{Kuperberg1997,Westbury2005} and it uses the theory of Kashiwara
crystals.

\begin{prop} 
Let $C$ be the crystal of the representation $V$. For $n\geqslant 0$,
$T_3(n)$ is the number of highest weight words of weight 0 in the tensor power
$\otimes^nC$.
\end{prop}

The highest weight words are in bijection with a set of excursions in the
weight lattice of $G_2$. The steps are the seven weights of the fundamental
representation $V$. The walk is constrained to stay in the dominant chamber.
These weights and the dominant chamber are shown in Figure~\ref{fig:stepswt}.
There is one further rule, which is the \emph{boundary rule}, stating that if
a walk is at a boundary point $(0,y)$, then the step $(0,0)$ is not permitted.

Let $E_3(n)$ be the $n$-th term of the second octant sequence~\oeis{A108307}.
Then the sequence $E_3$ enumerates hesitating tableaux of height 2, introduced
in \cite{MR2272140}. The excursion interpretation is given in
\cite{Bousquet-Melou2005}.

A \Dfn{hesitating tableau} of semilength $n$ is an excursion in the Young
lattice with $n$ steps. Each step is one of the following pairs of moves on
the Young lattice:
\begin{itemize}
    \item do nothing, add a cell;
    \item remove a cell, do nothing;
    \item add a cell, remove a cell.
\end{itemize}

A hesitating tableau of \Dfn{height~$h$} is a hesitating tableau such that
every partition in the sequence has height at most $h$. A hesitating tableau
of height 2 can be interpreted as an excursion in $\bZ^2$ by identifying
partitions with at most two nonzero rows with the set
\[ \{(x,y)\in\bZ^2 \; | \; x\geqslant y\geqslant 0. \} \]
There are eight steps since there are two ways to do nothing then add a cell,
two ways to remove a cell then do nothing and four ways to remove a cell then
add a cell. Two of the four ways to remove a cell then add a cell give the
step $(0,0)$, namely add a cell on the first row then remove this cell and add
a cell on the second row then remove this cell. It is always allowed to add
and then remove a cell on the first row. The step which adds and removes a
cell on the second row is not permitted on the line $x=y$.

\begin{proof}[Proof of Theorem~\ref{thm:bt}]

We compare the two descriptions of the walks. The steps of the walks for the
sequence $T_3$ are shown in Figure~\ref{fig:stepswt} and the steps of the
walks for the sequence $E_3$ are shown in Figure~\ref{fig:stepswk}.

In order to compare these, we first make the change of coordinates $(x,y) \to
(x,y\pm x)$. This identifies the six non-zero steps, as well as the two
domains.

By Lemma~\ref{LEM:binomialwalks}, it remains to compare the zero steps. There
is one zero step in the~$T_3$ case and two in the $E_3$ case. In the $E_3$
case we have the zero step which adds and then removes a cell on the second
row. The boundary condition is that this step is not allowed on the line
$x=y$. After the change of coordinates, this is the same boundary condition as
the zero step in the $T_3$ case. The second zero step in the $E_3$ case adds
and then removes a cell in the first row. This is always allowed. 

\end{proof}

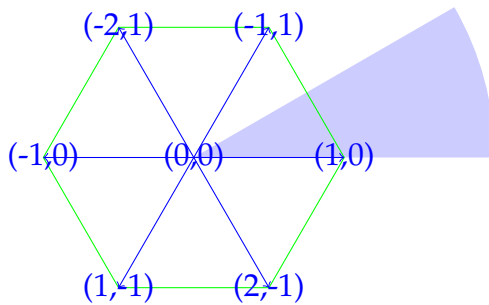
\begin{figure}
    \centering
\begin{tikzpicture}[scale=2]
    \fill[blue!20] (0,0) -- (2,0) arc(0:30:2) -- cycle;
    \draw[green] (1,0) -- (60:1) -- (120:1) -- (-1,0) -- (240:1) -- (300:1) -- cycle;
    \draw[->,blue] (0,0) -- (1,0) node {(1,0)};
    \draw[->,blue] (0,0) -- (60:1) node {(-1,1)};
    \draw[->,blue] (0,0) -- (120:1) node {(-2,1)};
    \draw[->,blue] (0,0) -- (-1,0) node {(-1,0)};
    \draw[->,blue] (0,0) -- (240:1) node {(1,-1)};
    \draw[->,blue] (0,0) -- (300:1) node {(2,-1)};
    \draw[blue] (0,0) node {(0,0)};
\end{tikzpicture}
    \caption{Steps in weight lattice}
    \label{fig:stepswt}
\end{figure}

\begin{figure}
    \centering
\begin{tikzpicture}[scale=1.7]
    \fill[blue!20] (0,0) -- (2,0) arc(0:45:2) -- cycle;
    \draw[green] (1,1) -- (1,-1) -- (-1,-1) -- (-1,1) -- cycle;
    \draw[->,blue] (0,0) -- (1,0) node {(1,0)};
    \draw[->,blue] (0,0) -- (0,1) node {(0,1)};
    \draw[->,blue] (0,0) -- (-1,1) node {(-1,1)};
    \draw[->,blue] (0,0) -- (-1,0) node {(-1,0)};
    \draw[->,blue] (0,0) -- (0,-1) node {(0,-1)};
    \draw[->,blue] (0,0) -- (1,-1) node {(1,-1)};
    \draw[blue] (0,0) node {(0,0)};
\end{tikzpicture}
    \caption{Steps in octant}\label{fig:stepswk}
\end{figure}
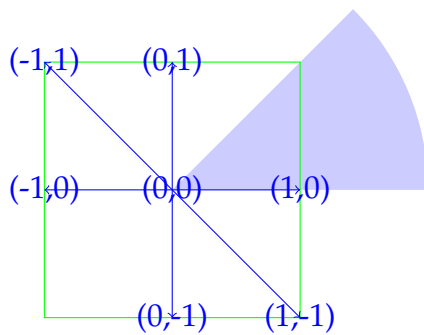

\subsection{Algebra} \label{SUBSEC:algebra}

In this subsection, we give three proofs of Theorem~\ref{thm:rc}, with
different flavors. The first two proofs use Theorem~\ref{thm:bt} and a result
of Bousquet-M\'elou and Xin~\cite{Bousquet-Melou2005} on partitions that avoid
3-crossings. The last proof relies on the connection with $G_2$ walks.

\subsubsection{First proof of Theorem~\ref{thm:rc}} 
The first proof is based on Theorem~\ref{thm:bt}, on~Proposition~2
in~\cite{Bousquet-Melou2005} and on the method of creative
telescoping~\cite{Zeilberger1991, KoutschanThesis} for the summation of
(bivariate) holonomic sequences.

\begin{prop}\label{main-thm-H}\cite[Proposition 2]{Bousquet-Melou2005}
The number $E_3(n)$ of partitions of $[n]$ having no enhanced $3$-crossing is
given by $ E_3\left( 0 \right)=E_3\left( 1 \right) =1,$ and for $n\ge 0$,
\begin{multline}
8\left( n+3 \right)  \left( n+1 \right) E_3 \left( n \right) +
 \left( 7{n}^{2}+53n+88 \right) E_3 \left( n+1 \right) \\
 - \left( n+8
 \right)  \left( n+7 \right) E_3\left( n+2 \right)=0.
\end{multline}
Equivalently, the associated generating function $\mathcal E(t)=\sum_{n\ge 0}
E_3(n)t^n$ satisfies
\begin{multline*}\label{e-dfinite-H}
{t}^{2} \left(1+ t \right)  \left(1- 8t \right) {\frac
{d^{2}}{d{t} ^{2}}}\mathcal E ( t )
+2t  \left( 6-23t-20{t}^{2} \right) {\frac {d}{dt}}\mathcal E ( t ) \\
+6 \left(5-7t -4{t}^{2}\right) \mathcal E ( t ) -30=0.
\end{multline*}
\end{prop}

\begin{proof}[First Proof of Theorem~\ref{thm:rc}]
By Theorem~\ref{thm:bt}, we have 
\[ T_3(n) = \sum_{k = 0}^n (-1)^{n - k} \binom nk E_3(k). \]
Set $f(n, k) = (-1)^{n - k} \binom nk E_3(k)$. By
Proposition~\ref{main-thm-H}, and by the closure properties for holonomic
functions, it follows that $f(n, k)$ is holonomic. Thus, we can apply Chyzak's
algorithm~\cite{KoutschanThesis} for creative telescoping to derive a
recurrence relation for~$T_3$. In particular, using the Koutschan's
Mathematica package {\tt HolonomicFunctions.m}~\cite{Christoph2010} that
implements Chyzak's algorithm, we find exactly the recurrence equation in
Theorem~\ref{thm:rc}.

\end{proof}
\noindent The detailed calculation involved in the above proof can be found
in~\cite{ElectronicYZ}.

\subsubsection{Second proof of Theorem~\ref{thm:rc}}

The second proof is also based on~Proposition~2 in~\cite{Bousquet-Melou2005}
and on Theorem~\ref{thm:bt}, namely on the relation between the generating
functions of $T_3(n)$ and of $E_3(n)$ implied by Theorem~\ref{thm:bt}.

\begin{proof}[Second Proof of Theorem~\ref{thm:rc}] Let $\mathcal T(t) = \sum_{n \geq 0} T_3(n) t^n$. 
By Theorem~\ref{thm:bt} and Lemma~\ref{lem:btg},
\[ \mathcal T(t) = \frac{1}{1 + t} \cdot \mathcal E\left( \frac{t}{1 + t} 
\right). \]
We know from Proposition~\ref{main-thm-H} a differential equation for
$\mathcal E\left( t \right)$. By (univariate) closure properties of D-finite
functions, we deduce a differential equation for $\mathcal T(t)$, and convert
it into a linear recurrence for $T_3(n)$, which is exactly the recurrence in
Theorem~\ref{thm:rc}. \end{proof}

\subsubsection{Third proof of Theorem~\ref{thm:rc}}

The third proof of Theorem~\ref{thm:rc} relies on~$G_2$ walks and the method
in~\cite{MR3588720}. First we define two elements $K$ and $W$ in the group
ring of the root lattice of $G_2$. Since the root lattice has rank 2, these
become Laurent polynomials in two variables once we choose a basis of the root
lattice. The element $K$ is the character of the representation $V$; $W$ is
given by
\[ W = \sum_{w\in W} \varepsilon(w)\left[ w(\rho)-\rho \right] \]
Here $W$ is the Weyl group, $\varepsilon\colon W\to \{\pm 1\}$ is the sign
character and $\rho$ is half the sum of the positive roots.

The following definition is given in \cite{MR1265145} in the paragraph
following Conjecture~3.3. The general principle is discussed in
\cite{MR1092920}.

\begin{defn}\label{def:W}
The $n$-th term $T_3(n)$ is the constant term of the Laurent polynomial
$W\,K^n$, where
\[ K = (1 + x + y + x\,y + x^{-1} + y^{-1} + (xy)^{-1}), \]
and $W$ is the Laurent polynomial
\begin{multline*}
W = x^{-2}y^{-3}(x^2y^3 - xy^3 + x^{-1}y^2 - x^{-2}y +x^{-3}y^{-1} - x^{-3}y^{-2}\\ + x^{-2}y^{-3} - x^{-1}y^{-3} + xy^{-2} - x^2y^{-1} + x^3y - x^3y^2).
\end{multline*}
\end{defn}

\begin{proof}[Third Proof of Theorem~\ref{thm:rc}] 
Let $\mathcal T(t) = \sum_{n \geq 0} T_3(n) t^n$. By Definition~\ref{def:W},
$\mathcal T(t)$ is the constant coefficient $[x^0 y^0]$ of $W/(1-tK)$. In
other words, $\mathcal T(t)$ is equal to the algebraic residue of
$W/(xy-txyK)$, and thus it is D-finite~\cite{MR929767} and is annihilated by a
linear differential operator that cancels the contour integral of
$W/(xy-txyK)$ over a cycle. Using the integration algorithm for multivariate
rational functions in~\cite{Dwork} we compute the following operator of
order~6, that cancels $\mathcal T(t)$:
\begin{align*}
L_6 = {t}^{5} \left( t+1 \right)  \left( 7\,t -1\right)  \left( 2\,t+1 \right) ^{2}\partial^{6}+ \\
3\,{t}^{4} \left( 2\,t+1 \right) 
 \left( 168\,{t}^{3}+211\,{t}^{2}+40\,t-11 \right) \partial^{5}+ \\
6\,{t}^{3} \left( 2100\,{t}^{4}+3475\,{t}^{3}+1616\,{t}^{2}+
79\,t-61 \right) \partial^{4}+ \\
6\,{t}^{2} \left( 11200\,{t}^{4}+17400\,{t}^{3}+7556\,{t}^{2}+268\,t-273 \right) \partial^{3}+ \\
36\,t \left( 4200\,{t}^{4}+6100\,{t}^{3}+2442\,{t}^{2}+54\,t-77 \right) \partial^{2} +\\
36\, \left( 3360\,{t}^{4}+4540\,{t}^{3}
+1646\,{t}^{2}+16\,t-35 \right) \partial+\\
20160\,{t}^{3}+25200\,{t}^{2}+8064\,t,
\end{align*}
where $\partial = \frac{\partial}{\partial t}$ denotes the derivation operator
with respect to~$t$.

The operator $L_6$ factors as $L_6 = Q L_3$, where
\[
Q
=
 \left( 2\,t+1 \right) {t}^{3}{\partial}^{3}+ \left( 24\,t+13 \right) {t}^{2}{\partial}^{2}+6\, \left( 12\,t+7 \right) t{\partial}+48\,t+30,
\]
and
\begin{align*}
L_3 =
{t}^{2} \left( 2\,t+1 \right)  \left( 7\,t -1\right)  \left( t+1 \right) {\partial}^{3}+2\,t \left( t+1 \right)  \left( 63\,{t
}^{2}+22\,t-7 \right) {\partial}^{2}+ \\ \left( 252\,{t}^{3}+338\,{t}^{2}+36\,t-42 \right) {\partial}+28\,t \left( 3\,t+4 \right).
\end{align*}
This shows that $f(t) := L_3(\mathcal T(t))$ is a solution of the differential
operator $Q$. Hence, by denoting $f(t) = \sum_{n \geq 0} f_n t^n$, one deduces
that for all $n\geq 0$ we have \[2 \left( n+2 \right) f_n + \left( n+6 \right)
f_{n+1} = 0. \] One the other hand, from $\mathcal T(t) = 1+t^2 + O(t^3)$, it
follows that~$f_0=0$, therefore $f(t) = 0$, in other words $\mathcal T(t)$ is
also a solution of~$L_3$. From there, deducing the recurrence relation of
Theorem~\ref{thm:rc} is immediate. 
\end{proof}

\subsection{Closed formulae} \label{SUBSEC:closedformulae}

By factorization of the operator $L_3$ and using algorithms for solving second order differential equations, as described
in~\cite{MR3588720}, we derive the following closed formula for $\mathcal T(t)$:
\begin{equation*}
\mathcal T(t) = 
\setlength\arraycolsep{1pt}
\frac {1}{30 \, t^5} \left[
R_1\cdot {}_2F_1\left(\begin{matrix}\frac{1}{3}& &\frac{2}{3}\\&2&
\end{matrix};\phi\right) +
R_2 \cdot {}_2 F_1\left(\begin{matrix}\frac{2}{3}& &\frac{4}{3}\\&3&
\end{matrix};\phi\right) + 5 \, P
\right],
\end{equation*} 
where 
\[
R_1 = {\frac { \left( t+1 \right) ^{2} \left( 214\,{t}^{3}+45\,{t}^{2}+60\,t+5 \right) }{t-1}}, \quad \quad R_2 = 
6\,{\frac { {t}^{2} \left( t+1 \right) ^{2} \left( 101\,{t}^{2}+74\,t+5 \right) }{ \left( t-1 \right) ^{2}}},
\]
and 
\[
\phi = \frac{27\,(t+1)\,t^2}{(1-t)^3}, \quad \quad \quad \quad  P=
{28\,{t}^{4}+66\,{t}^{3}+46\,{t}^{2}+15\,t+1}.
\]

Following the approach in~\cite[\S3.3]{MR3588720}, itself based
on~\cite[Theorem 1.5]{StienstraBeukers85}, one can obtain an alternative
expression using a more geometric flavor. The key point is that the
denominator of $W/(xy-txyK)$ is a family of \emph{elliptic} curves, thus
integrating $W/(xy-txyK)$ over a small torus amounts to computing the periods
of the two forms (of the first and second kind). Working out the details, this
approach yields an expression for $\mathcal T(t)$ in terms of the Weierstrass
invariant
\[ g_2 = \left( t-1 \right)  \left( 25\,{t}^{3} + 21\,{t}^{2} + 3\,t - 1 \right) \]
and of the $j$-invariant
\[
J = 
{\frac { \left( t-1 \right)^3  \left( 25\,{t}^{3} + 21\,{t}^{2} + 3\,t - 1 \right) ^{3}}{{t}^{6} \left( 1-7\,t \right)  \left( 2
\,t+1 \right) ^{2} \left( t+1 \right) ^{3}}}
\]
of our family of curves. As in~\cite{MR3588720}, we introduce the expression 
\[
\setlength\arraycolsep{1pt}
H(t)  = 
\frac {1}{{g_2}^{1/4}} \cdot 
{}_2 F_1\left(\begin{matrix}\frac{1}{12}& &\frac{5}{12}\\& 1 &
\end{matrix};\frac{1728}{J}\right).
\]
Then $\mathcal T(t)$ is proved to be equal to
\begin{align*}
	\frac{P}{6 \, t^5} + 
\frac{	\left( 7\,t - 1\right) 
	 \left( 2\,t+1 \right)  \left( t+1 \right)}{360 \, t^5} 
	   \Big(  \left( 155\,{t}^{2}+182\,t+59 \right) \left( 11\,t+1 \right)  {\it H} \left( t \right)+ \\
 \left( 341\,{t}^{3}+507\,{t}^
	{2}+231\,t+1 \right) \left( 5\,t+1 \right)  H'(t)
\Big).
\end{align*}

\section{Quadrant sequences}

The second family of sequences we are interested in is given in
Figure~\ref{fig:bx}. We show that sequences in this second family are
identical to sequences associated to representations of $G_2$ restricted to
$SL(3)$. Thus, by Lemma~\ref{LEM:binomialrep}, those sequences are also
related by binomial transforms. Moreover, they can also be interpreted as
lattice walks restricted to the quadrant. This is why we call them
\emph{quadrant sequences}. They are related to the octant sequences by the
branching rules~\cite{MR638077} for the maximal subgroup $SL(3)$ of $G_2$.

\begin{figure}
    \centering
\begin{tabular}{c|rrrrrr}
$\ba$ (\href{https://oeis.org}{OEIS} tag) 
& 0 & 1 & 2 & 3 & 4 & 5\\ \hline
\oeis{A151366} & 1 & 0 & 2 & 2 & 12 & 30 \\ 
\oeis{A236408} & 1 & 1 & 3 & 9 & 33 & 131 \\ 
\oeis{A001181} & 1 & 2 & 6 & 22 & 92 & 422 \\
\oeis{A216947} & 1 & 3 & 11 & 49 & 221 & 1113
\end{tabular}
\caption{The second family of sequences $\ba$,
their \href{https://oeis.org}{OEIS} tags
and their first terms.}    
\label{fig:bx}
\end{figure}

\begin{defn}\label{defn:quad}
Let $V$ be the defining representation of $SL(3)$ and denote the dual by
$V^*$. For $k\geqslant 0$, the quadrant sequences $\mathcal{S}_k$ are the
sequence associated to $(SL(3), V\oplus V^*\oplus k\;\bC)$.
\end{defn}
 
\begin{lemma}\label{lemma:quad2}
Let $C_k$ be the generating function of $\mathcal{S}_k$, where $k \geq 0$.
Then $C_k$ is the constant coefficient of $[x^0 y^0]$ of $W/(1-tK)$, where
\begin{equation} \label{EQ:K}
K = k+x+y+{x}^{-1}+{y}^{-1}+{\frac {x}{y}}+{\frac {y}{x}}
\end{equation}
and
\begin{equation}\label{EQ:W}
W = 1-{\frac {{x}^{2}}{y}}+{x}^{3}-{x}^{2}{y}^{2}+{y}^{3}-{\frac {{y}^{2}}{x}}.
\end{equation}
\end{lemma}

Let $C_2(n)$ be the $n$-th term of the sequence~\oeis{A216947} in the second
family. Marberg~\cite[Theorem 1.7]{Marberg2012} showed $C_2(n)$ is the
constant term $[x^0 y^0]$ of $W \tilde{K}^n$, where $\tilde{K} = K|_{k = 3}$,
the Laurent polynomials $K$ and $W$ are specified in~\eqref{EQ:K}
and~\eqref{EQ:W}.
By Lemma~\ref{lemma:quad2}, the sequence $\mathcal{S}_3$ is identical to  the sequence~\oeis{A216947}. 
Therefore, we have:

\begin{prop}~\cite[Theorem 1.7]{Marberg2012} \label{LEM:marberg}
The $n$-th term $C_2(n)$ of the sequence~$\mathcal{S}_3$ is given by $C_2(0) = 1, C_2(1) = 3$ and for $n \geq 0$: 
\begin{multline*}
(n + 5) (n + 6) \cdot C_2(n +2)  - 2 (5 n^2 + 36 n + 61) \cdot C_2(n + 1) \\
 + 9 (n+1) (n +4) \cdot C_2(n) = 0,
\end{multline*}
Equivalently, the associated generating function $\mathcal{C}(t) = \sum_{n \ge 0} C_2(n) t^n$ satisfies 
\begin{multline*}
72 \mathcal{C}(t)+4 (-61 + 117 t) {\frac {d}{dt}}\mathcal C(t) +2 (15 - 184 t + 234 t^2) {\frac {d^2}{dt^2}}\mathcal C(t) \\
+2 t (-6 + 7 t) (-1 + 9 t) {\frac {d^3}{dt^3}}\mathcal C(t) + (-1 + t) t^2 (-1 + 9 t) {\frac {d^4}{dt^4}}\mathcal C(t) = 0. 
\end{multline*}
\end{prop}

Next, for $k = 0, 1, 2, 3$, we prove a uniform recurrence equation for the
sequence $\mathcal{S}_k$. It is given by a single formula with $k$ as a
parameter. Moreover, we show that $\mathcal{S}_0, \mathcal{S}_1,
\mathcal{S}_2, \mathcal{S}_3$ are identical to sequences in the second family.

\begin{thm}
For $k = 0, 1, 2, 3$, the $n$-th term $a(n)$ of the sequence $\mathcal{S}_k$ satisfies the following recurrence equation:
\begin{multline}
(k - 9) (k - 1) k^2 (n + 1) (n + 2) a(n) +  \\
2 k (n + 2) (2 k^2 n - 15 k n + 8 k^2 + 9 n - 56 k + 36
) a(n + 1) + \\  
( 6 k^2 n^2 + 54 k^2 n - 30 k n^2 + 114 k^2 + 9 n^2 - 254 k n + 81 n - 510 k
+ 162 )
a(n + 2) \\ + 
 2 
( 2 k n^2 + 24 k n - 5 n^2 + 70 k - 56 n -153)
a(n + 3) \\ 
 + (n + 7) (n + 8) a(n + 4) = 0 \label{EQ:uniformrec}
\end{multline} 
\end{thm}
\begin{proof}
By Lemma~\ref{LEM:binomialrep}, sequences $\mathcal{S}_2, \mathcal{S}_1,
\mathcal{S}_0$ are the first, second, and third binomial transforms of $S_3$,
respectively. Thus, by Lemma~\ref{lem:btg}, the generating function of $S_k$
is
\[
\mathcal{A}(t) := \frac{1}{1 - k t} \cdot \mathcal{C} \left(\frac{t}{1- kt} \right) \ \ \text{ for } \ \ k = 0, 1, 2, 3,
\] 
where $\mathcal{C}(t)$ is the generating function of $S_3$. Regarding $k$ as a
parameter in the above expression, we deduce the differential equation for
$\mathcal A(t)$ in the claim by using Proposition~\ref{LEM:marberg} and
closure properties of D-finite functions. By converting the differential
equation for $\mathcal A(t)$, we get the corresponding recurrence equation for
the sequence $a(n)$, which is exactly the recurrence equation in the claim.
\end{proof}

In~\cite{Marberg2012}, Marberg proved that \oeis{A151366}, \oeis{A001181}, and
\oeis{A216947} are related by binomial transforms by using combinatorial
methods. Here, we give another proof by the representation theory of simple
Lie algebras.

\begin{cor} \label{COR:identical}
The sequences $S_0, S_1, S_2, S_3$ are identical to sequences in the second
family specified in Figure~\ref{fig:bx}. Moreover, sequences in the second
family are related by binomial transforms.
\end{cor}

\begin{proof}

In~\eqref{EQ:uniformrec}, by setting $k$ to $0, 1, 2, 3$, we find that the
corresponding recurrence operators are left multiples~\cite[page
618]{Chen2016} of those of~\oeis{A216947},~\oeis{A001181},~\oeis{A236408},
and~\oeis{A151366} specified in OEIS. It implies that $S_0, S_1, S_2, S_3$
satisfy the same recurrence equations as the sequences in the second family.
To verify they are the same sequences, we just need to check finitely many
initial terms. The details of the verification can be find
in~\cite{ElectronicYZ}. Since $S_k$'s are related by binomial transforms, so
are sequences in the second family.
\end{proof}

%% \printbibliography

% \bibliography{setp}
% \bibliographystyle{plain}

\end{document}